\documentclass[a4paper,12pt,twoside,leqno,final]{amsart}
\usepackage{amsmath}
\usepackage{amssymb}

\newtheorem{thm}{Theorem}[section]
\newtheorem{lem}[thm]{Lemma}

\newtheorem{prop}[thm]{Proposition}
\newtheorem{cor}[thm]{Corollary}

\newcommand{\C}{{\mathbb C}}
\newcommand{\D}{{\mathbb D}}
\newcommand{\R}{{\mathbb R}}
\newcommand{\T}{{\mathbb T}}
\newcommand{\Z}{{\mathbb Z}}
\newcommand{\N}{{\mathbb N}}

\newcommand{\bmo}{{\rm BMO}}
\newcommand{\bmoa}{{\rm BMOA}}

\newcommand{\Al}{A^\alpha}

\newcommand{\La}{\Lambda}

\renewcommand{\sb}{\subset}

\newcommand{\eps}{\varepsilon}

\newcommand{\f}{\frac}
\newcommand{\ov}{\overline}
\newcommand{\al}{\alpha}

\newcommand{\Ga}{\Gamma}
\newcommand{\ga}{\gamma}

\newcommand{\ze}{\zeta}
\renewcommand{\th}{\theta}
\newcommand{\si}{\sigma}
\newcommand{\ph}{\varphi}

\newcommand{\Om}{\Omega}

\newcommand{\const}{\text{\rm const}}

\newcommand{\tg}{T_{\bar g}}
\newcommand{\kbmo}{K_{*\th}}
\newcommand{\kinf}{K^\infty_\th}

\renewcommand{\qedsymbol}{$\square$}
\numberwithin{equation}{section}

\title[Smooth analytic functions and model subspaces]
{Smooth analytic functions\\ 
and model subspaces}
\dedicatory{Dedicated to the memory of Cora Sadosky}
\author{Konstantin M. Dyakonov}
\address{ICREA, BGSMath and Universitat de Barcelona, Departament de 
Matem\`atiques i Inform\`atica, Gran Via 585, E-08007 Barcelona, Spain}
\email{konstantin.dyakonov@icrea.cat}
\thanks{Supported in part by grants MTM2011-27932-C02-01, MTM2014-51834-P from El Ministerio de 
Econom\'ia y Competitividad (Spain) and grant 2014-SGR-289 from AGAUR (Generalitat de Catalunya).}

\begin{document}
\begin{abstract}
The main themes of this survey are as follows: (a) the canonical (Riesz--Nevanlinna) factorization in various 
classes of analytic functions on the disk that are smooth up to its boundary, and (b) model subspaces (i.e., 
invariant subspaces of the backward shift) in the Hardy spaces $H^p$ and in $\bmoa$. It is the interrelationship 
and a peculiar cross-fertilization between the two topics that we wish to highlight. 
\end{abstract}

\maketitle

\section{Introduction}

Our first topic in this survey is the multiplicative structure in spaces of {\it smooth analytic functions}. 
This phrase may sound somewhat redundant, if not downright confusing, since every analytic function is 
automatically smooth (in any reasonable sense) on its domain. The term becomes perfectly meaningful, 
though, if \lq\lq smooth" is interpreted as \lq\lq smooth up to the boundary". It is indeed the boundary 
smoothness of analytic functions that interests us here. 

\par Our functions will live on the disk $\D:=\{z\in\C:\,|z|<1\}$. 
Putting the smoothness issue aside (but only for a short while), let us now recall 
a bit of function theory on the disk. Suppose that $f$ is analytic on $\D$ and not too large 
near the unit circle $\T:=\partial\D$. Specifically, assume that $f$ lies in some {\it Hardy 
space} $H^p$ with $0<p\le\infty$. By definition, this means -- in addition to analyticity -- that 
$$\|f\|_{H^p}:=\sup_{0<r<1}\left(\int_\T|f(r\ze)|^pdm(\ze)\right)^{1/p}<\infty$$
if $0<p<\infty$, or $\|f\|_{H^\infty}:=\sup_\D|f|<\infty$ if $p=\infty$. 
Here and below, $m$ denotes the normalized arclength measure on $\T$. It is well known that 
$H^p$ functions have boundary values (nontangential limits) $m$-almost everywhere on $\T$. 
We may then identify $H^p$ with a subspace of $L^p=L^p(\T,m)$ bearing in mind that the above 
norm, $\|\cdot\|_{H^p}$, agrees on $H^p$ with the standard $L^p$-norm $\|\cdot\|_p$ over $\T$ 
(see \cite[Chapter II]{G}). When $0<p<1$, the two quantities should actually be called quasinorms 
rather than norms. 

\par For $f$ as above, the function $\ph:=|f|\big|_\T$ will satisfy $\ph\in L^p$ 
and $\log\ph\in L^1$. Moreover, these last two conditions characterize the moduli 
of $H^p$ functions on $\T$. Now, letting $u:=\log\ph$ and writing $\mathcal Pu$ for 
the harmonic extension (via the Poisson integral) of $u$ from $\T$ into $\D$, we define the 
{\it outer function} $\mathcal O_\ph$ as the (essentially unique) analytic function on $\D$ 
satisfying $\log|\mathcal O_\ph(z)|=\mathcal Pu(z)$. This done, we have $\mathcal O_\ph\in H^p$ 
and $|\mathcal O_\ph|=\ph$ a.\,e. on $\T$. The ratio $f/\mathcal O_\ph=:\th$ will then be 
an {\it inner function}; that is, $\th\in H^\infty$ and $|\th|=1$ a.\,e. on $\T$. Thus we arrive 
at the Canonical Factorization Theorem: the general form of an $f\in H^p$ is given by $f=\th F$, 
where $\th$ is inner and $F$ outer (so that $F=\mathcal O_\ph$ for some $\ph$ as above). 
A further factorization formula for inner functions allows us to express $\th$ canonically in terms 
of its zeros $\{a_n\}$ (these are only required to satisfy $\sum_n(1-|a_n|)<\infty$) and a certain 
singular measure $\mu$ on $\T$; see \cite[Chapter II]{G}. In summary, the original function $f\in H^p$ 
is fully described by the parameters $\ph$, $\{a_n\}$ and $\mu$ that emerge; and any choice 
of parameters gives rise to an $f\in H^p$ via factorization. 

\par The terms \lq\lq inner function" and \lq\lq outer function" were coined by Beurling. Why did he 
call them that? An amusing, but rather controversial, explanation I have heard is that the identity 
$f=\th F$, when written in {\it this} specific form, has $\th$ (the \lq\lq inner factor") inside 
and $F$ (the \lq\lq outer factor") outside. Observe that in some noncommutative generalizations, which 
we do not touch upon, the order may become crucial; and yes, it should be $\th F$ rather than $F\th$. 

\par While quite a bit of modern 1-D complex analysis has evolved in an attempt to extend the $H^p$ 
theory to {\it larger} analytic spaces, one also feels tempted to look at {\it smaller} (nicer) 
classes, in particular, at those populated by smooth analytic functions. Here, the good news is that 
the canonical factorization theorem applies. The bad news is, however, that the parameters cannot be 
chosen freely. Indeed, most inner functions -- actually, all the \lq\lq interesting" (i.\,e., nonrational) 
ones -- are highly oscillatory, hence discontinuous, at some points of $\T$. Consequently, the product 
$\th F$ may only be smooth on $\T$ if the outer factor, $F$, is good enough and kills the singularities 
of the (bad) inner factor, $\th$. To find an explicit quantitative expression of this interplay, for 
a given \lq\lq smooth analytic space", is therefore one problem to be dealt with. 

\par Our second topic is the {\it model subspaces}, alias {\it star-invariant subspaces}, in $H^p$ 
and in $\bmoa:=\bmo\cap H^1$, where $\bmo=\bmo(\T)$ is the space of functions of bounded mean oscillation 
on $\T$ (see \cite[Chapter VI]{G}). In $H^2$, the model subspace $K_\th$ generated by an inner function 
$\th$ is, by definition, the orthogonal complement of the shift-invariant subspace $\th H^2$. Thus, 
\begin{equation}\label{eqn:kate}
K_\th\left(=K^2_\th\right):=H^2\ominus\th H^2.
\end{equation}
It is a reproducing kernel Hilbert space, whose kernel function $k_z$ associated with a point $z\in\D$ is 
given by 
$$k_z(\ze)=\f{1-\ov{\th(z)}\th(\ze)}{1-\ov z\ze}.$$ 
This last function is therefore in $K_\th$ for every $z$, and every $f\in K_\th$ satisfies 
$$f(z)=\int_\T f(\ze)\ov{k_z(\ze)}\,dm(\ze),\qquad z\in\D.$$ 
It is straightforward to verify that $K_\th=H^2\cap\th\,\ov{H^2_0}$, and we further define $K^p_\th$ 
(the $H^p$-analogue of $K_\th$) by putting 
$$K^p_\th:=H^p\cap\th\,\ov{H^p_0},\qquad1\le p\le\infty,$$ 
where $H^p_0:=\{f\in H^p:f(0)=0\}$ and the bar denotes complex conjugation. For smaller 
$p$'s, a more reasonable definition appears to be 
$$K^p_\th:=\text{\rm clos}_{H^p}K_\th,\qquad0<p<1.$$ 
These subspaces play a crucial role in the Sz.-Nagy--Foia\c{s} operator 
model (see \cite{N}), which accounts for the terminology. 
Now, the term \lq\lq star-invariant" means invariant 
under the backward shift operator $f\mapsto(f-f(0))/z$, and it 
follows from Beurling's theorem (see \cite[Chapter II]{G}) that 
the general form of a closed and nontrivial star-invariant subspace 
in $H^2$ is indeed given by \eqref{eqn:kate}, with $\th$ inner. 
A similar fact is true for $H^p$ when $1\le p<\infty$. 
\par Finally, we put  
$$\kbmo:=K_\th\cap\bmoa.$$ 
When equipped with the $\bmo$-norm $\|\cdot\|_*$, $\kbmo$ becomes a star-invariant subspace of $\bmoa$; 
in fact, it is the annihilator in $\bmoa$ of the shift-invariant subspace $\th H^1$ in $H^1$. 
Of course, $\kbmo$ contains $\kinf$ and is contained in every $K^p_\th$ with $0<p<\infty$. 
\par While each of the two topics just mentioned has received quite a bit of attention in its own 
right, the intimate interconnection between them does not seem to have been noticed (until recently) 
or explored in any detail. It is precisely the systematic exploitation of this interrelationship, perhaps 
a kind of duality, between the two subjects that is characteristic of our approach. In fact, the 
three stories told in the next three sections are intended to show that results and methods pertaining 
to one of our themes cast new light on the other, and vice versa. 

\par Before moving any further, we need to recall the notions of Toeplitz and Hankel operators, since 
these will be crucial in what follows. We let $P_+$ and $P_-$ denote the orthogonal projections from 
$L^2$ onto $H^2$ and onto $\ov{H^2_0}$, respectively. Thus, 
$$(P_+F)(z):=\sum_{n\ge0}\widehat F(n)z^n\quad\text{\rm and}\quad (P_-F)(z):=\sum_{n<0}\widehat F(n)z^n,$$ 
where $\widehat F(n):=\int_\T F(\ze)\ov\ze^ndm(\ze)$ is the $n$th Fourier coefficient of $F$. 
These operators are then extended to $L^p$ with $1<p<\infty$ (in which case they become bounded 
projections onto $H^p$ and $\ov{H^p_0}$, the classical M. Riesz theorem tells us) and furthermore 
to $L^1$ (even though $P_{\pm}(L^1)\not\sb L^1$). Next, given a measurable function $\psi$ on $\T$, we write 
$$T_\psi f:=P_+(\psi f)\quad\text{\rm and}\quad H_\psi f:=P_-(\psi f),$$ 
whenever $f\in H^1$ and $\psi f\in L^1$. The mapping $T_\psi$ (resp., $H_\psi$) is called the {\it Toeplitz} 
(resp., {\it Hankel}) {\it operator with symbol} $\psi$. 

\par In the special case where $\psi$ is analytic (i.e., $\psi\in H^1$), $T_\psi$ reduces to the multiplication 
map $f\mapsto f\psi$, defined at least on $H^\infty$. The Toeplitz operators with symbols in $\ov{H^1}$ are 
said to be {\it coanalytic}. It is also worth mentioning that the model subspace $K^p_\th$ (where $p\ge1$) 
or $\kbmo$, with $\th$ an inner function, is precisely the kernel of the coanalytic Toeplitz operator $T_{\ov\th}$ 
acting on $H^p$ or $\bmoa$. 

\par Because Toeplitz and Hankel operators were among Cora Sadosky's best beloved mathematical creatures, 
their appearance in this survey seems to be appropriate (and is, anyway, far from incidental to the subject 
matter). 

\par We conclude this introduction with a brief outline of the rest of the paper. In Sections 2 and 3, we look 
at certain smooth analytic spaces $X$ and seek to characterize the pairs $(f,\th)$, with $f\in X$ and $\th$ inner, 
which satisfy 
\begin{equation}\label{eqn:prodfth}
f\th\in X.
\end{equation}
Sometimes it is more natural to replace \eqref{eqn:prodfth} by 
\begin{equation}\label{eqn:prodforallk}
f\th^k\in X\,\,\,\text{\rm for all}\,\,\,k\in\N,
\end{equation}
and we are led to consider some other related conditions as well. In Section 2, the role of $X$ is played 
by the analytic Lipschitz--Zygmund spaces $\Al$ (see the beginning of that section for definitions), and the 
pairs $(f,\th)$ with property \eqref{eqn:prodforallk} are then explicitly described by a certain {\it smallness 
condition}, to be imposed on $|f|$ near the singularities of $\th$. Furthermore, the same smallness condition 
ensures that the multiplication operator $g\mapsto fg$ acts nicely on the model space $K^p_\th$, or perhaps 
on $K^p_{\th^n}$ with $n$ suitably large, by improving integrability properties of the functions therein. For 
instance, given $1<p<q<\infty$ and $\al=p^{-1}-q^{-1}$, we prove that multiplication by a function $f\in\Al$ 
maps $K^p_\th$ into $H^q$ if and only if it maps $\th$ into $\Al$ (so that \eqref{eqn:prodfth} holds with $X=\Al$). 
The case of smaller $p$'s and larger $\al$'s leads to a minor complication involving \eqref{eqn:prodforallk} in 
place of \eqref{eqn:prodfth}, and $K^p_{\th^n}$ in place of $K^p_\th$. 

\par In Section 3, our space $X$ is chosen from among the so-called Dirichlet-type spaces. Each of these is formed 
by the functions $f\in H^2$ whose coefficient sequence, $\{\widehat f(n)\}$, lies in a certain weighted $\ell^2$. 
An important special case is the classical {\it Dirichlet space} $\mathcal D$, the set of analytic functions 
$f$ on $\D$ whose derivative, $f'$, is square integrable over $\D$ with respect to the normalized area 
measure $A$; the (semi)norm $\|f\|_{\mathcal D}$ is then defined to be $\left(\int_\D|f'|^2dA\right)^{1/2}$. 
Among other things we recover, for $f\in\mathcal D$ and $\th$ inner, the identity 
\begin{equation}\label{eqn:carlident}
\|f\th\|^2_{\mathcal D}=\|f\|^2_{\mathcal D}+\int_\T|f|^2|\th'|dm,
\end{equation}
which forms part of Carleson's celebrated formula from \cite{Carl}. Moreover, we obtain similar -- but more 
sophisticated -- formulas for general Dirichlet-type spaces; these yield the smallness conditions on $f$ 
(in relation to $\th$) that are responsible for the interplay between the two factors in \eqref{eqn:prodfth}, 
for the current choices of $X$. When $X=\mathcal D$, the corresponding smallness condition reads 
$\int_\T|f|^2|\th'|dm<\infty$, as readily seen from \eqref{eqn:carlident}. Our approach to \eqref{eqn:carlident} 
is based on the fact that the quantity $\|f\th\|_{\mathcal D}$ coincides with the Hilbert--Schmidt norm of 
the Hankel operator $H_{\ov{f\th}}$ acting from $H^2$ to $\ov{H^2_0}$ (and similarly for $f$ in place of $f\th$). 
Now let $\{g_n\}$ be an orthonormal basis in the model subspace $K_\th$. Since $H^2=\th H^2\oplus K_\th$, the 
family $\{\th z^k\}_{k\ge0}\cup\{g_n\}$ is an orthonormal basis in $H^2$, and we may use it to compute the 
Hilbert--Schmidt norm of $H_{\ov{f\th}}$. This gives 
$$\|f\th\|^2_{\mathcal D}=\sum_{k\ge0}\left\|H_{\ov{f\th}}(\th z^k)\right\|^2_2
+\sum_n\left\|H_{\ov{f\th}}g_n\right\|^2_2,$$ 
and a further calculation shows that the two sums above reduce to the two terms on the right-hand side of 
\eqref{eqn:carlident}. A modification of the same technique allows us to handle the case of a generic 
Dirichlet-type space. 

\par In Section 4, we consider coanalytic Toeplitz operators on the model subspace $\kbmo$, and 
we obtain a criterion for such an operator to act boundedly from $\kbmo$ to a given analytic 
space $X$, under certain assumptions on the latter. Precisely speaking, the spaces $X$ that arise 
here naturally are those which enjoy the {\it $K$-property} of Havin. In other words, it will be 
assumed that every Toeplitz operator $T_{\ov h}$ with $h\in H^\infty$ maps $X$ boundedly into 
itself and satisfies $\|T_{\ov h}\|_{X\to X}\le\const\cdot\|h\|_\infty$. This property was introduced 
by Havin in \cite{H}, where he also verified it for a number of smooth analytic spaces. (It was further 
observed in \cite{H} that every space $X$ with the $K$-property admits division by inner factors: 
whenever $f\in X$ and $I$ is an inner function such that $f/I\in H^1$, it follows that $f/I\in X$.) 
Now, the appearance of the $K$-property in connection with model subspaces of $\bmoa$ seems to reveal 
yet another link between the two topics of concern. 

\par The content of Section 2 is essentially borrowed from the author's papers \cite{DSpb93,DAJM}, while 
Sections 3 and 4 are based on \cite{DSpb97} and \cite{DIUMJ}, respectively. It seems that a bit of 
self-plagiarism is unavoidable -- and hopefully pardonable -- under the circumstances. 

\section{Factorization in Lipschitz--Zygmund spaces}

This section deals with the {\it Lipschitz--Zygmund spaces} $\La^\al=\La^\al(\T)$ and their analytic 
subspaces $\Al$. For $0<\al<\infty$, the space $\La^\al$ is defined as the set of all (complex-valued) 
functions $f\in C(\T)$ that satisfy 
\begin{equation}\label{eqn:deflipzyg}
\|\Delta_h^nf\|_\infty=O(|h|^\al),\qquad h\in\R,
\end{equation}
where $\|\cdot\|_\infty$ is the $\sup$-norm on $\T$, $n$ is an integer with $n>\al$, 
and $\Delta_h^n$ denotes the $n$th order difference operator with step $h$. (As usual, the difference 
operators $\Delta_h^k$ are defined by induction: one puts $(\Delta_h^1f)(\ze):=f(e^{ih}\ze)-f(\ze)$ and 
$\Delta_h^kf:=\Delta_h^1\Delta_h^{k-1}f$.) It is well known that property \eqref{eqn:deflipzyg} does not 
depend on the choice of $n$, as long as $n>\al$, except possibly for the constant in the $O$-condition. 

\par The corresponding analytic subspaces are 
$$\Al:=\La^\al\cap H^\infty,\qquad0<\al<\infty.$$ 
Equivalently, by a theorem essentially due to Hardy and Littlewood, $\Al$ is formed by those holomorphic 
functions $f$ on $\D$ which obey the condition 
$$|f^{(n)}(z)|=O\left((1-|z|)^{\al-n}\right),\qquad z\in\D,$$
for some (and then every) integer $n$ with $n>\al$; here $f^{(n)}$ is the $n$th order derivative of $f$. 
The spaces $\La^\al$ and $\Al$ are then normed in a natural way. 

\par The main result of this section is Theorem \ref{thm:factlipzyg} below, which characterizes the pairs 
$(f,\th)$, with $f\in\Al$ and $\th$ inner, such that $f$ admits multiplication and/or division by every 
power of $\th$ in $\La^\al$. The characterization involves an explicit quantitative condition saying that 
$|f(z)|$ must decay at a certain rate as $z$ approaches the boundary along the sublevel set 
\begin{equation}\label{eqn:defsublevel}
\Om(\th,\eps):=\{z\in\D:\,|\th(z)|<\eps\}
\end{equation}
with $0<\eps<1$. Moreover, it turns out that the same decay condition provides a criterion for the multiplication 
operator $T_f:\,g\mapsto fg$ to map the model subspace $K^p_{\th^n}$ continuously into $H^q$, once the exponents 
are related appropriately. 

\begin{thm}\label{thm:factlipzyg} Suppose that $0<p<\infty$, $\max(1,p)<q<\infty$, $\al=p^{-1}-q^{-1}$, and 
$n$ is an integer with $np>1$. Assume also that $f\in\Al$ and $\th$ is an inner function.  
\par The following conditions are equivalent: 
\par{\rm (i)} $f\ov\th^k\in\La^{\al}$ for all $k\in\N$. 
\par{\rm (ii)} $f\ov\th^n\in\La^{\al}$. 
\par{\rm (iii)} The multiplication operator $T_f$ maps $K^p_{\th^n}$ boundedly into $H^q$. 
\par{\rm (iv)} For some (or every) $\eps\in(0,1)$, one has
\begin{equation}\label{eqn:decaycond}
|f(z)|=O((1-|z|)^\al)\quad\text{for}\quad z\in\Om(\th,\eps).
\end{equation}
\par{\rm (v)} $f\th^k\in\Al$ for all $k\in\N$. 
\par{\rm (vi)} $f\th^n\in\Al$. 
\end{thm}

\par It should be noted that the set $\Om(\th,\eps)$ hits $\T$ precisely at those points 
which are singular for $\th$. Thus, \eqref{eqn:decaycond} tells us how strongly the good 
factor $f$ must vanish on the bad set of the problematic (nonsmooth) factor $\th$ in order 
that the products in question be appropriately smooth. 

\par Postponing the proof for a while, we first establish a few preliminary facts to lean upon. To begin with, 
we recall the Duren--Romberg--Shields theorem (see \cite{DRS}) which allows us to identify $\Al$ with the dual 
of the Hardy space $H^r$, where $r=(1+\al)^{-1}$, under the pairing 
$$\langle\ph,\psi\rangle=\int_\T\ph\ov\psi\,dm.$$ 
For a given $\psi\in\Al$, the integral above is well defined at least when $\ph\in H^\infty$, and we have 
$$|\langle\ph,\psi\rangle|\le c_\al\|\ph\|_r\|\psi\|_{\La^{\al}}$$
with some constant $c_\al>0$. Moreover, the norm of the functional induced by $\psi$ on $H^r$ is actually 
comparable to $\|\psi\|_{\La^{\al}}$. 
\par The next three lemmas exploit this duality relation. The first of these was established by Havin in 
\cite{H}; we also cite Shamoyan \cite{Sha} in connection with part (b) below. 

\begin{lem}\label{lem:havin} Let $0<\al<\infty$. 
\par{\rm (a)} If $h\in H^\infty$, then the Toeplitz operator $T_{\ov h}$ maps 
the space $\Al$ boundedly into itself, with norm at most $\const\cdot\|h\|_\infty$. 
\par{\rm (b)} If $f\in H^1$ and $\th$ is an inner function such that $f\th\in\Al$, then $f\in\Al$ and 
$\|f\|_{\La^{\al}}\le\const\cdot\|f\th\|_{\La^{\al}}$. 
\par The constants are allowed to depend only on $\al$. 
\end{lem}

\par In Havin's terminology, statements (a) and (b) can be rephrased by saying that $\Al$ has the $K$-property 
and the (weaker) $f$-property, respectively. To prove (a), one notes that $T_{\ov h}$ is the adjoint 
of the multiplication operator $T_h:\,g\mapsto gh$, which is obviously bounded on $H^r$ with norm at 
most $\|h\|_\infty$. To deduce (b) from (a), observe that $f=T_{\ov\th}(f\th)$. 

\begin{lem}\label{lem:hankelhphq} Suppose that $0<p<\infty$, $\max(1,p)<q<\infty$, and $\al=p^{-1}-q^{-1}$. 
If $f\in\Al$, then the Hankel operator $H_{\ov f}$, defined by 
$$H_{\ov f}g=P_-(\ov fg),\qquad g\in H^\infty,$$ 
can be extended to a bounded linear operator mapping $H^p$ into $\ov H^q_0$.
\end{lem}

\begin{proof} Put $r=(1+\al)^{-1}$ and $q'=q/(q-1)$. Given $g\in H^\infty$ and $h\in H_0^{q'}$, we have 
\begin{equation*}
\begin{aligned}
\left|\int_\T\left(H_{\ov f}g\right)h\,dm\right|&=\left|\int_\T P_-(\ov fg)\cdot h\,dm\right|
=\left|\int_\T\ov fgh\,dm\right|\\
&\le c_\al\|f\|_{\La^\al}\|gh\|_r\le c_\al\|f\|_{\La^\al}\|g\|_p\|h\|_{q'}.
\end{aligned}
\end{equation*}
Here, the last two inequalities rely on the Duren--Romberg--Shields duality theorem and on H\"older's 
inequality. Taking the supremum over the unit-norm functions $h$ in $H_0^{q'}$, we obtain 
$$\|H_{\ov f}g\|_q\le c_\al\|f\|_{\La^\al}\|g\|_p,$$
which proves the required result. 
\end{proof}

\begin{lem}\label{lem:toepktheta} Suppose that $0<p<\infty$, $\max(1,p)<q<\infty$, and $\al=p^{-1}-q^{-1}$. 
Further, let $f\in H^2$ and let $\th$ be an inner function. If $P_-(f\ov\th)\in\La^\al$, then the operator 
$T_{\ov f}\big|_{K_\th^\infty}$ can be extended to a bounded linear operator acting from $K_\th^p$ to $H^q$. 
\end{lem}

\begin{proof} Given $g\in K_\th^\infty$, put $h:=\bar z\bar g\th$ (so that $h\in H^\infty$) and $\psi:=P_-(f\ov\th)$. 
The elementary identity 
$$\ov{P_+F}=zP_-(\ov z\ov F),\qquad F\in L^2,$$
shows that $\ov{T_{\ov f}g}=zH_\psi h$. Using Lemma \ref{lem:hankelhphq}, we get 
$$\|T_{\ov f}g\|_q=\|H_\psi h\|_q\le\const\cdot\|\psi\|_{\La^\al}\|h\|_p=\const\cdot\|\psi\|_{\La^\al}\|g\|_p,$$
which completes the proof.
\end{proof}

\par As a final preliminary result, we list some facts about the so-called Carleson curves associated with an 
inner function; see \cite[Chapter VIII]{G} for a proof. 

\begin{lem}\label{lem:carleson} Given an inner function $\th$ and a number $\eps\in(0,1)$, there exists a countable 
(possibly finite) system $\Ga_\eps=\Ga_\eps(\th)$ of simple closed rectifiable curves in $\D\cup\T$ with the 
following properties. 
\par{\rm (a)} The interiors of the curves in $\Ga_\eps$ are pairwise disjoint; the intersection of each of these 
curves with the circle $\T$ has zero length. 
\par{\rm (b)} One has $\eta<|\th|<\eps$ on $\Ga_\eps\cap\D$ for some positive $\eta=\eta(\eps)$. 
\par{\rm (c)} The arclength $|dz|$ on $\Ga_\eps\cap\D$ is a Carleson measure, i.e., $H^1\subset L^1(\Ga_\eps,|dz|)$; 
moreover, the norm of the corresponding embedding operator is bounded by a constant $N(\eps)$ depending only 
on $\eps$. 
\par{\rm (d)} For every $F\in H^1$, the equality 
$$\int_\T\f F\th dz =\int_{\Ga_\eps}\f F\th dz$$
holds true, provided that the curves in the family $\Ga_\eps$ are oriented appropriately. 
\end{lem}

\par Now we are in a position to prove our main result in this section. 

\medskip\noindent{\it Proof of Theorem \ref{thm:factlipzyg}.} The implications (i)$\implies$(ii) and (v)$\implies$(vi) 
being obvious, our plan is to show that (ii)$\implies$(iii)$\implies$(iv)$\implies$(i)$\implies$(v) and also 
that (vi)$\implies$(iii). 

\par (ii)$\implies$(iii). Write $u:=\th^n$ and let $g\in K^\infty_u$. Note that 
\begin{equation}\label{eqn:toha}
\ov fg=T_{\ov f}g+H_{\ov f}g.
\end{equation}
Since $f\in\Al$, Lemma \ref{lem:hankelhphq} tells us that 
$$\|H_{\ov f}g\|_q\le c_\al\|f\|_{\La^\al}\|g\|_p.$$
Now, since $f\ov u\in\La^\al$ by (ii), it follows that $P_-(f\ov u)\in\La^\al$ (indeed, the operators $P_+$ and $P_-$ 
are known to map $\La^\al$ into itself), and Lemma \ref{lem:toepktheta} gives 
$$\|T_{\ov f}g\|_q\le c_\al\|P_-(f\ov u)\|_{\La^\al}\|g\|_p.$$
The last two inequalities, together with \eqref{eqn:toha}, imply 
$$\|\ov fg\|_q\le\const\cdot\|g\|_p,$$ 
where the constant does not depend on $g$. Obviously, 
$$\|T_fg\|_q=\|fg\|_q=\|\ov fg\|_q,$$
and since $K^\infty_u$ is dense in $K^p_u$, we conclude that $T_f$ is a bounded operator from $K^p_u$ to $H^q$. 

\par (iii)$\implies$(iv). Fix $z\in\D$ and consider the reproducing kernel $k_z$ (for $K^2_\th$), given by 
$$k_z(\ze)=\f{1-\ov{\th(z)}\th(\ze)}{1-\ov z\ze}.$$
Since $k_z\in K^\infty_\th$, it follows that $k^n_z\in K^\infty_{\th^n}(\subset K^p_{\th^n})$; indeed, 
$$k^n_z\ov{\th^n}=\left(k_z\ov\th\right)^n\in\ov{H_0^\infty}.$$
Therefore, by (iii), 
\begin{equation}\label{eqn:kerpq}
\|fk^n_z\|_q\le\const\cdot\|k^n_z\|_p.
\end{equation}
In order to derive further information from this inequality, we now estimate its right-hand side from above, and 
the left-hand side from below. The elementary estimate 
$$\int_\T\f{dm(\ze)}{|\ze-z|^\ga}\le\f{C_\ga}{(1-|z|)^{\ga-1}}\qquad(\ga>1)$$
shows that 
\begin{equation}\label{eqn:rightabove}
\begin{aligned}
\|k^n_z\|_p&=\left(\int_\T\left|\f{1-\ov{\th(z)}\th(\ze)}{1-\ov z\ze}\right|^{np}dm(\ze)\right)^{1/p}\\
&\le2^n\left(\int_\T\f{dm(\ze)}{|\ze-z|^{np}}\right)^{1/p}\le\f{\const}{(1-|z|)^{n-1/p}},
\end{aligned}
\end{equation}
since $np>1$. 
\par Now let $F$ stand for the outer factor of $f$. Using the Cauchy integral formula, we get 
\begin{equation}\label{eqn:leftbelow}
\begin{aligned}
\|fk^n_z\|_q&=
\left(\int_\T|F(\ze)|^q\left|\f{1-\ov{\th(z)}\th(\ze)}{1-\ov z\ze}\right|^{nq}dm(\ze)\right)^{1/q}\\
&\ge\left|\int_\T F^q(\ze)\f{(1-\ov{\th(z)}\th(\ze))^{nq}}{(1-\ov z\ze)^{nq-1}}\f{dm(\ze)}{1-z\ov\ze}\right|^{1/q}\\
&=\left(|F(z)|^q\f{(1-|\th(z)|^2)^{nq}}{(1-|z|^2)^{nq-1}}\right)^{1/q}
=|F(z)|\f{(1-|\th(z)|^2)^n}{(1-|z|^2)^{n-1/q}}\\
&\ge\const\cdot|f(z)|\f{(1-|\th(z)|)^n}{(1-|z|)^{n-1/q}}.
\end{aligned}
\end{equation}
In view of \eqref{eqn:rightabove} and \eqref{eqn:leftbelow}, inequality \eqref{eqn:kerpq} now yields 
$$|f(z)|\cdot(1-|\th(z)|)^n\le\const\cdot(1-|z|)^{1/p-1/q}=\const\cdot(1-|z|)^\al,$$
the constant being independent of $z$. Hence, for $0<\eps<1$, we have 
$$|f(z)|\le\const\cdot(1-\eps)^{-n}(1-|z|)^\al$$
whenever $z\in\Om(\th,\eps)$, so that (iv) holds true. 

\par (iv)$\implies$(i). We begin by showing that if (iv) is fulfilled with some $\eps\in(0,1)$, 
then $f\ov\th\in\La^{\al}$. Since 
$$f\ov\th=T_{\ov\th}f+H_{\ov\th}f$$
and $T_{\ov\th}f\in\Al$ (recall Lemma \ref{lem:havin}), it suffices to check that $H_{\ov\th}f\in\La^{\al}$. To 
this end, we take an arbitrary function $g\in H_0^\infty$ with $\|g\|_r=1$, where $r=(1+\al)^{-1}$, and verify 
that the integrals $\int_\T(H_{\ov\th}f)g\,dm$ are bounded in modulus by a constant independent of $g$. This 
will mean that the function $zH_{\ov\th}f$ generates a continuous linear functional on $H^r$, and hence lies 
in $\ov{\Al}$. Writing $g_1:=g/z$ and using the Carleson curves $\Ga_\eps=\Ga_\eps(\th)$ as described in 
Lemma \ref{lem:carleson}, we obtain 
\begin{equation*}
\begin{aligned}
\left|\int_\T(H_{\ov\th}f)g\,dm\right|&=\left|\int_\T f\ov\th g\,dm\right|=
\left|\f1{2\pi i}\int_\T\f{fg_1}\th\,dz\right|\\
&=\left|\f1{2\pi i}\int_{\Ga_\eps}\f{fg_1}\th\,dz\right|
\le\f1{2\pi}\int_{\Ga_\eps}\f{|f||g_1|^{1-r}|g_1|^r}{|\th|}|dz|.
\end{aligned}
\end{equation*}
Because $g_1$ is a unit-norm function in $H^r$, it follows easily that $|g_1(z)|^r\le(1-|z|)^{-1}$, whence 
$$|g_1(z)|^{1-r}\le(1-|z|)^{-(1-r)/r}=(1-|z|)^{-\al},\qquad z\in\D.$$ 
Plugging this into the preceding estimate and recalling that $|\th|\ge\eta(\eps)$ on $\Ga_\eps\cap\D$, we find 
that 
\begin{equation}\label{eqn:estintcar}
\left|\int_\T(H_{\ov\th}f)g\,dm\right|\le\f1{2\pi\eta(\eps)}\cdot
\left(\sup_{z\in\Ga_\eps\cap\D}\f{|f(z)|}{(1-|z|)^\al}\right)\cdot\int_{\Ga_\eps}|g_1|^r\,|dz|.
\end{equation}
Since $\Ga_\eps\cap\D$ is contained in $\Om(\th,\eps)$, the supremum in \eqref{eqn:estintcar} is finite 
by virtue of (iv). Also, 
$$\int_{\Ga_\eps}|g_1|^r\,|dz|\le N(\eps)\cdot\int_\T|g_1|^r\,dm=N(\eps).$$ 
Taking this into account, we deduce from \eqref{eqn:estintcar} that 
$$\sup\left\{\left|\int_\T(H_{\ov\th}f)g\,dm\right|:\,g\in H_0^\infty,\,\|g\|_r=1\right\}
\le\f{CN(\eps)}{2\pi\eta(\eps)},$$
where $C$ is the constant coming from the $O$-condition in (iv). This means that $H_{\ov\th}f\in\La^\al$, and 
hence $f\ov\th\in\La^\al$. 
\par Replacing $\th$ by $\th^k$ and $\eps$ by $\eps^k$ in the above argument, we similarly verify that 
$f\ov{\th^k}\in\La^\al$ for every $k\in\N$. 

\par (i)$\implies$(v). Assuming (i), we prove first that $f\th\in A^\al$, or equivalently, that 
\begin{equation}\label{eqn:harlitprod}
(f\th)^{(n)}(z)=O((1-|z|)^{\al-n})\quad\text{\rm as}\quad|z|\to1^-. 
\end{equation}
For $z\in\D$ and almost all $\ze\in\T$, we have the elementary identity 
$$\th^{n+1}(\ze)=(\th(\ze)-\th(z))^{n+1}+\sum_{k=0}^n\ph_k(z)\th^k(\ze),$$
where 
$$\ph_k(z):=(-1)^{n-k}\binom{n+1}{k}\th^{n+1-k}(z).$$
Therefore, 
\begin{equation*}
\begin{aligned}
(f\th)^{(n)}(z)&=\f{n!}{2\pi i}\int_\T\f{f(\ze)\th(\ze)}{(\ze-z)^{n+1}}d\ze
=\f{n!}{2\pi i}\int_\T\f{f(\ze)\ov{\th^n(\ze)}\th^{n+1}(\ze)}{(\ze-z)^{n+1}}d\ze\\
&=\f{n!}{2\pi i}\int_\T(f\ov\th^n)(\ze)\left(\f{\th(\ze)-\th(z)}{\ze-z}\right)^{n+1}d\ze
+\f{n!}{2\pi i}\sum_{k=0}^n\ph_k(z)\int_\T\f{f(\ze)\ov\th^{n-k}(\ze)}{(\ze-z)^{n+1}}d\ze\\
&=\f{n!}{2\pi i}\int_\T(f\ov\th^n)(\ze)\cdot\Phi_z(\ze)d\ze
+\sum_{k=0}^n\ph_k(z)\cdot\left(T_{\ov\th^{n-k}}f\right)^{(n)}(z),
\end{aligned}
\end{equation*}
where 
$$\Phi_z(\ze):=\left(\f{\th(\ze)-\th(z)}{\ze-z}\right)^{n+1}.$$
\par In view of (i), $f\ov\th^{n-k}\in\La^\al$ for $k=0,\dots,n$, so that $T_{\ov\th^{n-k}}f\in\Al$, which 
implies that 
$$\left(T_{\ov\th^{n-k}}f\right)^{(n)}(z)=O((1-|z|)^{\al-n})\quad\text{\rm as}\quad|z|\to1^-.$$
The functions $\ph_k(z)$ are bounded in $\D$, and to prove \eqref{eqn:harlitprod} it remains to verify that 
\begin{equation}\label{eqn:remains}
\int_\T(f\ov\th^n)(\ze)\cdot\Phi_z(\ze)\f{d\ze}{2\pi i}=O((1-|z|)^{\al-n})\quad\text{\rm as}\quad|z|\to1^-.
\end{equation}
Denote the integral on the left-hand side by $I_n(z)$. Since $\Phi_z\in H^\infty$, it follows that 
$$|I_n(z)|=\left|\int_\T(f\ov\th^n)(\ze)\cdot\ze\Phi_z(\ze)\,dm(\ze)\right|
\le c_\al\|P_-(f\ov\th^n)\|_{\La^\al}\|\Phi_z\|_r;$$
here, as before, $r=(1+\al)^{-1}$. Because $n>\al$, we have $(n+1)r>1$ and 
$$\|\Phi_z\|_r\le2^{n+1}\left(\int_\T\f{dm(\ze)}{|\ze-z|^{(n+1)r}}\right)^{1/r}
\le\f{\const}{(1-|z|)^{n+1-1/r}}=\f{\const}{(1-|z|)^{n-\al}},$$
where the constant does not depend on $z$. Consequently, 
$$|I_n(z)|\le\const\cdot\|P_-(f\ov\th^n)\|_{\La^\al}(1-|z|)^{\al-n}.$$
Since 
$$\|P_-(f\ov\th^n)\|_{\La^\al}\le C_\al\|f\ov\th^n\|_{\La^\al}<\infty$$
by virtue of (i), the estimate \eqref{eqn:remains} is thereby established. 
\par Thus, we have proved the implication 
$$(f\in\Al)\,\,\&\,\,\text{\rm(i)}\implies f\th\in\Al.$$ 
Applying this inductively to $f\th$, $f\th^2$, etc., in place of $f$, we eventually deduce from (i) that 
$f\th^k\in\Al$ for each $k\in\N$. 

\par (vi)$\implies$(iii). Write $u:=\th^n$ and suppose that $g\in K_u^\infty$. Then $g\ov u\in\ov{H_0^\infty}$, 
and hence 
$$\ov f\ov ug=P_-(\ov f\ov ug)=H_{\ov f\ov u}g.$$
Therefore, 
\begin{equation}\label{eqn:piglet}
\|fg\|_q=\|\ov f\ov ug\|_q=\|H_{\ov f\ov u}g\|_q\le c_\al\|fu\|_{\La^\al}\|g\|_p,
\end{equation}
where the last inequality is due to Lemma \ref{lem:hankelhphq}. The quantity $\|fu\|_{\La^\al}$ is finite in view 
of (vi), and \eqref{eqn:piglet} tells us that 
$$\|fg\|_q\le\const\cdot\|g\|_p$$
with a constant independent of $g$. Thus, the multiplication operator $T_f:g\mapsto fg$ maps $K^p_u$ boundedly 
into $H^q$, as required.\quad\qed

\medskip If we wish to restrict ourselves to the issue of multiplying or dividing a function $f\in\Al$ by 
an inner function $\th$ (and its powers), leaving out the model subspace part, we may state the result in 
a more concise form as follows. 

\begin{prop}\label{prop:concise} Suppose that $0<\al<\infty$, $n\in\N$, and $n>\al$. Given $f\in\Al$ and an 
inner function $\th$, the four statements below are equivalent. 
\par{\rm (i)} $f\th^n\in\Al$.
\par{\rm (ii)} $f\ov\th^n\in\La^{\al}$.
\par{\rm (iii)} $f\th^k\in\La^{\al}$ for all $k\in\Z$.
\par{\rm (iv)} Condition \eqref{eqn:decaycond} holds for some (or every) $\eps\in(0,1)$. 
\end{prop}

\par To prove this, it suffices to choose exponents $p$ and $q$ (once $\al$ and $n$ are given) so as to make 
the hypotheses of Theorem \ref{thm:factlipzyg} true, and then invoke the theorem. 

\medskip\noindent{\it Remarks.} (1) An alternative route to Proposition \ref{prop:concise} (but not to Theorem 
\ref{thm:factlipzyg} in its entirety) via the pseudoanalytic extension method was found by Dyn'kin \cite{Dyn}. 
A similar technique was later used by the author in \cite{DActa} to completely characterize the functions 
in $\Al$, $0<\al<1$, and in more general Lipschitz-type spaces, in terms of their moduli. (In particular, some 
equivalent forms of the crucial condition \eqref{eqn:decaycond} came out as a corollary.) Subsequently, 
Pavlovi\'c \cite{P} gave a more elementary proof of that result from \cite{DActa}. 

\smallskip (2) Some of the conditions in Theorem \ref{thm:factlipzyg} and Proposition \ref{prop:concise} would 
become simpler if we could take $n=1$. This can be done if $1<p<\infty$ in Theorem \ref{thm:factlipzyg}, or 
if $0<\al<1$ in Proposition \ref{prop:concise}, but not in the general case. Indeed, it follows from Shirokov's 
work (see \cite{ShiLNM, Shi}) that for each $\al>1$, one can find $f\in\Al$ and a Blaschke product $\th$ such 
that $f/\th\in\Al$, but $f\th\not\in\Al$. This means, in particular, that conditions (i) and (ii) in Proposition 
\ref{prop:concise} are no longer equivalent when $\al>1$ and $n=1$. The equivalence does hold under certain 
additional assumptions, though; these are likewise discussed in \cite{ShiLNM, Shi}. See also \cite{DScand, DSpb10} 
for an alternative study of this phenomenon. 

\smallskip (3) Given $\al\in(0,\infty)\setminus\Z$, suppose that $f\in\Al$ and $\th$ is an inner function. Comparing 
our Proposition \ref{prop:concise} with Shirokov's earlier results (see \cite{ShiTr, ShiLNM, Shi}), one infers 
that condition \eqref{eqn:decaycond} holds if and only if 
\begin{equation}\label{eqn:shircond}
m(\si(\th))=0\qquad\&\qquad|f(\ze)|=O\left(\f1{|\th'(\ze)|^\al}\right)\,\,\text{\rm for }\ze\in\T\setminus\si(\th),
\end{equation}
where $\si(\th)$ is the set of boundary singularities for $\th$. The equivalence between \eqref{eqn:decaycond} 
and \eqref{eqn:shircond} was also verified directly in \cite[Section 2]{DSpb93}. 

\smallskip (4) Theorem \ref{thm:factlipzyg} and Proposition \ref{prop:concise} remain valid in the case $\al=0$ 
(with $n=1$ and $1<p=q<\infty$), provided that the spaces $\Lambda^0$ and $A^0$ are taken to be $\bmo$ and $\bmoa$, 
respectively. This convention might be justified by the duality relations $\Al=(H^{1/(1+\al)})^*$ and 
$\bmoa=(H^1)^*$. The $\bmo$ versions of the above results are discussed in more detail in \cite[Section 5]{DSpb93}. 

\smallskip (5) In \cite{DSpb10}, we also considered the algebra $H^\infty_n:=\{f:\,f^{(n)}\in H^\infty\}$, $n\in\N$, 
in place of $\Al$, and we came up with an analogue of Proposition \ref{prop:concise} in that context.

\section{Factorization in Dirichlet-type spaces}

For a sequence $w=\{w_k\}_{k=1}^\infty$ of nonnegative numbers, the corresponding {\it Dirichlet-type space} 
$\mathcal D_w$ is formed by those functions $f\in H^2$ for which the quantity 
\begin{equation}\label{eqn:defdirnorm}
\|f\|_w:=\left(\sum_{k=1}^\infty w_k|\widehat f(k)|^2\right)^{1/2}
\end{equation}
is finite. The case $w_k=k$ corresponds to the classical {\it Dirichlet space} $\mathcal D(=\mathcal D_{\{k\}})$, 
the set of all functions $f\in H^2$ with 
$$\|f\|_\mathcal D:=\left(\int_\D|f'(z)|^2dA(z)\right)^{1/2}<\infty$$ 
(here $A$ is the normalized area measure on $\D$), and we have $\|\cdot\|_\mathcal D=\|\cdot\|_{\{k\}}$. 
\par We begin by establishing a certain orthogonality relation involving Toeplitz operators on Dirichlet-type 
spaces. 

\begin{thm}\label{thm:orthdir} Given numbers $0\le w_1\le w_2\le\dots$, let $w=\{w_k\}_{k=1}^\infty$ and let 
$\ga=\{\ga_k\}_{k=1}^\infty$ be the sequence defined by 
\begin{equation}\label{eqn:defgamma}
\ga_1=w_1,\qquad\ga_k=w_k-w_{k-1}\quad(k=2,3,\dots).
\end{equation}
Suppose that $F\in H^2$, $\th$ is an inner function, and $\{g_n\}$ is an orthonormal basis in $K_\th$. 
If $\Phi:=zT_{\ov z\ov\th}F$ and $h_n:=zT_{\ov\th}(Fg_n)$, then 
\begin{equation}\label{eqn:orthrel}
\|F\|_w^2=\|\Phi\|_w^2+\sum_n\|h_n\|_\ga^2
\end{equation}
(the definition of $\|\cdot\|_\ga$ being similar to \eqref{eqn:defdirnorm} above).
\end{thm}

\par To keep on the safe side, we remark that sequences with unspecified index sets, which we occasionally 
employ, are allowed to be finite (and sometimes empty). In particular, the orthonormal basis $\{g_n\}$ in 
Theorem \ref{thm:orthdir} will be finite if and only if $\th$ is a finite Blaschke product. 

\par The proof will make use of the notion of a {\it Hilbert--Schmidt operator}. Recall that, given two separable 
Hilbert spaces $\mathcal H_1$ and $\mathcal H_2$, a linear operator $T:\mathcal H_1\to\mathcal H_2$ is said to be 
Hilbert--Schmidt if the quantity 
$$\|T\|_{\mathfrak S_2}:=\left(\sum_n\|Te_n\|_{\mathcal H_2}^2\right)^{1/2}$$ 
is finite for some (or each) orthonormal basis $\{e_n\}$ of $\mathcal H_1$. It is well known -- and easily 
shown -- that this quantity does not actually depend on the choice of $\{e_n\}$ and is therefore well 
defined. The set of all Hilbert--Schmidt operators from $\mathcal H_1$ to $\mathcal H_2$ is denoted 
by $\mathfrak S_2(\mathcal H_1,\mathcal H_2)$. 

\par Also, we need a lemma that relates Hilbert--Schmidt operators to Dirichlet-type spaces. We state and prove it 
now, before proceeding with the proof of Theorem \ref{thm:orthdir}. 

\begin{lem}\label{lem:multhankel} Let $F\in H^2$. Suppose that $w=\{w_k\}_{k=1}^\infty$ and 
$\ga=\{\ga_k\}_{k=1}^\infty$ are two sequences of nonnegative numbers related by 
\begin{equation}\label{eqn:partialsum}
w_n=\sum_{k=1}^n\ga_k\quad(n=1,2,\dots).
\end{equation}
Finally, consider the multiplier map $M_\ga$ acting by the rule    
\begin{equation}\label{eqn:defmultmap}
M_\ga\left(\sum_{k=1}^\infty a_k\ov z^k\right):=\sum_{k=1}^\infty\sqrt{\ga_k}a_k\ov z^k,\qquad z\in\T 
\end{equation}
(defined initially on the set of antianalytic trigonometric polynomials $\sum_ka_k\ov z^k$). Then the operator 
$M_\ga H_{\ov F}$ belongs (or has an extension belonging) to $\mathfrak S_2(H^2,\ov{H_0^2})$ if and only if 
$F\in\mathcal D_w$. Moreover, 
\begin{equation}\label{eqn:normmulthank}
\|M_\ga H_{\ov F}\|_{\mathfrak S_2}=\|F\|_w.
\end{equation}
\end{lem}

\begin{proof} Since $\{z^n\}_{n=0}^\infty$ is an orthonormal basis in $H^2$, we have 
\begin{equation}\label{eqn:kapusta}
\|M_\ga H_{\ov F}\|^2_{\mathfrak S_2}=\sum_{n=0}^\infty\|M_\ga H_{\ov F}z^n\|^2_2,
\end{equation}
where $\|\cdot\|_2$ is the usual $L^2$-norm. Letting $a_n:=\widehat F(n)$, we find that 
$$H_{\ov F}z^n=\sum_{k=1}^\infty\ov a_{n+k}\ov z^k,$$
whence 
$$M_\ga H_{\ov F}z^n=\sum_{k=1}^\infty\sqrt{\ga_k}\ov a_{n+k}\ov z^k,$$
and, by the Parseval identity, 
$$\|M_\ga H_{\ov F}z^n\|^2_2=\sum_{k=1}^\infty\ga_k|a_{n+k}|^2.$$
Plugging this into \eqref{eqn:kapusta} and recalling \eqref{eqn:partialsum}, we obtain 
\begin{equation*}
\begin{aligned}
\|M_\ga H_{\ov F}\|^2_{\mathfrak S_2}&=\sum_{n=0}^\infty\sum_{k=1}^\infty\ga_k|a_{n+k}|^2
=\sum_{j=1}^\infty|a_j|^2\sum_{k=1}^j\ga_k\\
&=\sum_{j=1}^\infty w_j|a_j|^2=\|F\|^2_w,
\end{aligned}
\end{equation*}
which proves \eqref{eqn:normmulthank} and the lemma. 
\end{proof}

\medskip\noindent{\it Proof of Theorem \ref{thm:orthdir}.} Let $M_\ga$ be the multiplier map defined by 
\eqref{eqn:defmultmap}. From Lemma \ref{lem:multhankel} we know that 
\begin{equation}\label{eqn:morkovka}
\|F\|^2_w=\|M_\ga H_{\ov F}\|^2_{\mathfrak S_2}.
\end{equation}
Consider the functions $G_n$ defined (a.e. on $\T$) by $G_n:=\bar z\bar g_n\th$. Since $\{g_n\}$ is an orthonormal 
basis in $K_\th$, the same is true for $\{G_n\}$ (indeed, the map $f\mapsto\bar z\bar f\th$ is an antilinear 
isometry of $K_\th$ onto itself). Furthermore, since $H^2=\th H^2\oplus K_\th$, the family $\{\th z^n\}_{n=0}^\infty
\cup\{G_n\}$ forms an orthonormal basis in $H^2$, and we may use it to compute the Hilbert--Schmidt norm 
in \eqref{eqn:morkovka}. In this way we obtain 
\begin{equation}\label{eqn:baklazhan}
\|M_\ga H_{\ov F}\|^2_{\mathfrak S_2}=\sum_{n=0}^\infty\|M_\ga H_{\ov F}(\th z^n)\|^2_2
+\sum_n\|M_\ga H_{\ov F}G_n\|^2_2=S_1+S_2,
\end{equation}
where $S_1$ and $S_2$ denote the two preceding sums, in the same order. The elementary identity 
\begin{equation}\label{eqn:elemid} 
P_-\ph=\bar z\ov{P_+(\bar z\bar\ph)},\qquad\ph\in L^2,
\end{equation}
yields 
$$P_-(\ov F\th)=\bar z\ov{P_+(\bar zF\bar\th)}=\ov\Phi,$$
whence 
\begin{equation*}
\begin{aligned}
H_{\ov F}(\th z^n)&=P_-(\ov F\th z^n)=P_-(P_-(\ov F\th)\cdot z^n)\\
&=P_-(\ov\Phi z^n)=H_{\ov\Phi}z^n.
\end{aligned}
\end{equation*}
Thus, 
\begin{equation}\label{eqn:sone}
S_1=\sum_{n=0}^\infty\|M_\ga H_{\ov\Phi}z^n\|^2_2=\|M_\ga H_{\ov\Phi}\|^2_{\mathfrak S_2}=\|\Phi\|^2_w,
\end{equation}
where the last equality relies on Lemma \ref{lem:multhankel}. 
\par Another application of \eqref{eqn:elemid} gives 
$$H_{\ov F}G_n=P_-(\ov F\bar z\bar g_n\th)=\bar z\ov{P_+(Fg_n\bar\th)}=\ov h_n,$$
and so 
$$\|M_\ga H_{\ov F}G_n\|^2_2=\sum_{k=1}^\infty\ga_k\left|\widehat{(H_{\ov F}G_n)}(-k)\right|^2=
\sum_{k=1}^\infty\ga_k|\widehat h_n(k)|^2=\|h_n\|^2_\ga.$$ 
Summing over $n$, we get 
\begin{equation}\label{eqn:stwo}
S_2=\sum_n\|M_\ga H_{\ov F}G_n\|^2_2=\sum_n\|h_n\|^2_\ga.
\end{equation}
Finally, we plug the identities coming from \eqref{eqn:sone} and \eqref{eqn:stwo} into \eqref{eqn:baklazhan}. 
Together with \eqref{eqn:morkovka}, this yields the required formula \eqref{eqn:orthrel}. \quad\qed

\medskip As a consequence of Theorem \ref{thm:orthdir}, we now deduce a result of Korenblum and Fa\u\i vyshevski\u\i 
\,\,concerning the action of certain Toeplitz operators on Dirichlet-type spaces. (In all fairness, their original 
theorem in \cite{KF} gives a bit more than our Corollary \ref{cor:korfa} below. Alternative routes to that result 
can be found in \cite{K} and \cite{R}.) To state it, we need a minor modification of the $\|\cdot\|_w$ norm. 
Namely, given a sequence $v=\{v_n\}_{n=0}^\infty$ of positive numbers and a holomorphic function 
$f(z)=\sum_{n=0}^\infty\widehat f(n)z^n$ on $\D$, we put 
$$\|f\|_{v,0}:=\left(\sum_{n=0}^\infty v_n|\widehat f(n)|^2\right)^{1/2}$$
(note that the value $n=0$ is now included). 

\begin{cor}\label{cor:korfa} Let $v=\{v_n\}_{n=0}^\infty$ be a nondecreasing sequence of positive numbers, and 
let $\th$ be an inner function. Then, for every $f,g\in H^2$, we have 
\begin{equation}\label{eqn:toepdir}
\|T_{\ov\th}f\|_{v,0}\le\|f\|_{v,0}
\end{equation}
and 
\begin{equation}\label{eqn:divdir}
\|g\|_{v,0}\le\|g\th\|_{v,0}.
\end{equation}
\end{cor}

\begin{proof} Put $F:=zf$ and define $\Phi$ as in Theorem \ref{thm:orthdir}, so that 
$$\Phi=zT_{\ov z\ov\th}F=zT_{\ov\th}f.$$
For $n=1,2,\dots$, let $w_n=v_{n-1}$ and $w=\{w_n\}_{n=1}^\infty$. Theorem \ref{thm:orthdir} implies that 
$\|\Phi\|_w\le\|F\|_w$. Observing that $\|\Phi\|_w=\|T_{\ov\th}f\|_{v,0}$ and $\|F\|_w=\|f\|_{v,0}$, we arrive 
at \eqref{eqn:toepdir}. To prove \eqref{eqn:divdir}, it suffices to apply \eqref{eqn:toepdir} with $f=g\th$. 
\end{proof} 

\par The next fact is likewise a straightforward consequence of Theorem \ref{thm:orthdir}. 

\begin{thm}\label{thm:orthdirbis} Let $w=\{w_k\}_{k=1}^\infty$ be a nondecreasing sequence with $w_1\ge0$, and 
let $\ga=\{\ga_k\}_{k=1}^\infty$ be defined by \eqref{eqn:defgamma}. If $f\in H^2$, $\th$ is an inner function, 
and $\{g_n\}$ is an orthonormal basis in $K_\th$, then 
\begin{equation}\label{eqn:orthrelbis}
\|f\th\|^2_w=\|f\|^2_w+\sum_n\|zfg_n\|^2_\ga.
\end{equation}
\end{thm}

\begin{proof} Put $F:=f\th$, and define $\Phi$ and $h_n$ as in Theorem \ref{thm:orthdir}. We have then 
$$\Phi=zT_{\ov z\ov\th}(f\th)=zT_{\ov z}f=f-f(0),$$
whence $\|\Phi\|_w=\|f\|_w$. Also, 
$$h_n=zT_{\ov\th}(f\th g_n)=zfg_n.$$ 
The formula \eqref{eqn:orthrel} therefore reduces to \eqref{eqn:orthrelbis}, and the proof is complete. 
\end{proof}

\par In some special cases, Theorem \ref{thm:orthdirbis} can be used to derive a more explicit form of the 
(nonnegative) \lq\lq discrepancy term" 
\begin{equation}\label{eqn:discrep}
R_w(f,\th):=\|f\th\|^2_w-\|f\|^2_w. 
\end{equation}
One such case is pointed out in Theorem \ref{thm:orthdirmom} below. Before stating the result, we need to recall 
some basic facts about angular derivatives. 

\par Given a function $\ph\in H^\infty$ with $\|\ph\|_\infty=1$, we say that $\ph$ has an {\it angular derivative} 
(in the sense of Carath\'eodory) at a point $\ze\in\T$ if both $\ph$ and $\ph'$ have nontangential limits 
at $\ze$, the former of these being of modulus $1$. (The two limits are then denoted by $\ph(\ze)$ and $\ph'(\ze)$, 
respectively.) The classical Julia--Carath\'eodory theorem (see \cite[Chapter VI]{B}, \cite[Chapter I]{Car} 
or \cite[Chapter VI]{Sar2}) asserts that this happens if and only if 
\begin{equation}\label{eqn:liminf}
\liminf_{z\to\ze}\f{1-|\ph(z)|^2}{1-|z|^2}<\infty.
\end{equation}
And if \eqref{eqn:liminf} holds, the theorem tells us also that $\ph'(\ze)$ coincides with the limit 
of the difference quotient 
$$\f{\ph(z)-\ph(\ze)}{z-\ze}$$ 
as $z\to\ze$ nontangentially. Moreover, $|\ph'(\ze)|$ will then agree with the value of the 
(unrestricted) $\liminf$ in \eqref{eqn:liminf}, and this remains true if $\liminf$ is replaced 
by the corresponding nontangential limit. 

\par Finally, if $\th=BS$ is an inner function (with $B$ a Blaschke product and $S$ singular), then 
\begin{equation}\label{eqn:angderinn}
|\th'(\ze)|=\sum_j\f{1-|a_j|^2}{|\ze-a_j|^2}+2\int_\T\f{d\mu(\eta)}{|\ze-\eta|^2},\qquad\ze\in\T,
\end{equation}
where $\{a_j\}$ is the zero sequence of $B$ and $\mu$ is the singular measure associated with $S$. This 
formula can be found in \cite{AC}; it holds for {\it every} point $\ze$ of $\T$, with the convention that 
$|\th'(\ze)|=\infty$ whenever $\th$ fails to possess an angular derivative at $\ze$. 

\begin{thm}\label{thm:orthdirmom} Let $\si$ be a positive Borel measure on $[0,1]$ with 
$\int_{[0,1]}x^2d\si(x)<\infty$. Put 
$$\ga_k:=\int_{[0,1]}x^{2k}d\si(x),\qquad k=1,2,\dots,$$
and define the sequence $w=\{w_n\}_{n=1}^\infty$ by \eqref{eqn:partialsum}. If $f\in H^2$ and $\th$ is an inner 
function, then 
\begin{equation}\label{eqn:carltypeorthrel}
\|f\th\|^2_w=\|f\|^2_w+\int_\T dm(\ze)\int_{[0,1]}r^2|f(r\ze)|^2\f{1-|\th(r\ze)|^2}{1-r^2}d\si(r).
\end{equation}
Here the value of $(1-|\th(r\ze)|^2)/(1-r^2)$ at $r=1$ is interpreted as $|\th'(\ze)|$, the modulus of the angular 
derivative of $\th$ at $\ze$. 
\end{thm}

\par The proof will rely on Theorem \ref{thm:orthdirbis} and on the following lemma. 

\begin{lem}\label{lem:sumbasis} Let $\th$ be an inner function, and let $\{g_n\}$ be an orthonormal basis in $K_\th$. 
Then 
\begin{equation}\label{eqn:suminside}
\sum_n|g_n(z)|^2=\f{1-|\th(z)|^2}{1-|z|^2},\qquad z\in\D.
\end{equation}
Furthermore, if $\ze\in\T$ is a point at which the limits $\lim_{r\to1^-}g_n(r\ze)=:g_n(\ze)$ exist for all $n$, then 
\begin{equation}\label{eqn:sumboundary}
\sum_n|g_n(\ze)|^2=|\th'(\ze)|.
\end{equation}
\end{lem}

\par To prove the lemma, consider the reproducing kernel 
$$k_z(w)=\f{1-\ov{\th(z)}\th(w)}{1-\ov zw}$$ 
of $K_\th$ and use Parseval's identity to get 
$$\sum_n|g_n(z)|^2=\sum_n|\langle g_n,k_z\rangle|^2=\|k_z\|_2^2=k_z(z)=\f{1-|\th(z)|^2}{1-|z|^2}$$ 
for $z\in\D$. This yields \eqref{eqn:suminside}, which in turn implies \eqref{eqn:sumboundary} upon putting 
$z=r\ze$ and passing to the limit as $r\to1^-$. 

\medskip\noindent{\it Proof of Theorem \ref{thm:orthdirmom}.} We may assume that $f\in\mathcal D_w$, since 
otherwise both sides of \eqref{eqn:carltypeorthrel} equal $\infty$. By Theorem \ref{thm:orthdirbis}, 
the \lq\lq discrepancy term" \eqref{eqn:discrep} is given by 
\begin{equation}\label{eqn:discrepbasis}
R_w(f,\th)=\sum_n\|zfg_n\|^2_\ga,
\end{equation}
where $\ga=\{\ga_k\}_{k=1}^\infty$ and $\{g_n\}$ is some (no matter which) orthonormal basis in $K_\th$. 
This said, we proceed by considering two special cases. 

\smallskip{\it Case 1:} $\si$ has no atom at $1$. We may think of the disk 
$$\D=\{r\ze:\,r\in[0,1),\,\ze\in\T\}$$
as of a measure space endowed with the product measure $\si\times m=:\nu$. The monomials $z^k$ ($k=1,2,\dots$) 
are then mutually orthogonal in $L^2(\D)$ and have norms $\sqrt{\ga_k}$. Therefore, for a function 
$h(z)=\sum_{k=1}^\infty\widehat h(k)z^k$ in $zH^1$, we have 
$$\|h\|^2_{L^2(\D,\nu)}=\sum_{k=1}^\infty\ga_k|\widehat h(k)|^2=\|h\|^2_\ga.$$ 
Applying this to $h_n:=zfg_n$ gives 
$$\|h_n\|^2_\ga=\|h_n\|^2_{L^2(\D,\nu)}=\int_\T dm(\ze)\int_{[0,1]}r^2|f(r\ze)|^2|g_n(r\ze)|^2d\si(r).$$ 
Consequently, in view of \eqref{eqn:discrepbasis}, 
\begin{equation}\label{eqn:discrepintegral}
R_w(f,\th)=\sum_n\|h_n\|^2_\ga=\int_\T dm(\ze)\int_{[0,1]}r^2|f(r\ze)|^2\sum_n|g_n(r\ze)|^2d\si(r). 
\end{equation}
By Lemma \ref{lem:sumbasis}, 
$$\sum_n|g_n(r\ze)|^2=\f{1-|\th(r\ze)|^2}{1-r^2},$$
and so \eqref{eqn:discrepintegral} reduces to 
$$R_w(f,\th)=\int_\T dm(\ze)\int_{[0,1]}r^2|f(r\ze)|^2\f{1-|\th(r\ze)|^2}{1-r^2}d\si(r),$$
which proves \eqref{eqn:carltypeorthrel}. 

\smallskip{\it Case 2:} $\si$ is the unit point mass at $1$. In this case, we have $\ga_k=1$ and $w_k=k$, so 
that $\|\cdot\|_\ga=\|\cdot\|_2$ on $zH^2$, and $\|\cdot\|_w=\|\cdot\|_{\mathcal D}$. Therefore, we can rewrite 
\eqref{eqn:discrepbasis} in the form 
$$\|f\th\|^2_{\mathcal D}-\|f\|^2_{\mathcal D}=\sum_n\|zfg_n\|^2_2=\int_\T|f(\ze)|^2\sum_n|g_n(\ze)|^2dm(\ze).$$
Combining this with \eqref{eqn:sumboundary}, we finally obtain 
\begin{equation}\label{eqn:discrepdir}
\|f\th\|^2_{\mathcal D}-\|f\|^2_{\mathcal D}=\int_\T|f(\ze)|^2|\th'(\ze)|dm(\ze),
\end{equation}
which coincides with \eqref{eqn:carltypeorthrel} under the current hypothesis on $\si$. 
\par The general case being a combination of Cases 1 and 2, the required result follows.\quad\qed

\medskip\noindent{\it Remark.} Recalling the identity \eqref{eqn:angderinn} and plugging it 
into \eqref{eqn:discrepdir}, we find that 
\begin{equation}\label{eqn:carlform}
\|f\th\|^2_{\mathcal D}=\|f\|^2_{\mathcal D}+
\int_\T|f(\ze)|^2\left(\sum_j\f{1-|a_j|^2}{|\ze-a_j|^2}+2\int_\T\f{d\mu(\eta)}{|\ze-\eta|^2}\right)dm(\ze)
\end{equation}
(here, as before, $\{a_j\}$ is the zero sequence of $\th$, and $\mu$ is the associated singular measure). 
This was established by Carleson in \cite{Carl}. In fact, the formula given there is a combination of 
\eqref{eqn:carlform} and an explicit expression for the Dirichlet integral $\|f\|^2_{\mathcal D}$ of an 
outer function $f$.

\section{Model subspaces in $\bmoa$} 

\par It has been noticed that various smoothness properties of an inner function $\th$, if available, 
tend to be inherited (typically, in a weaker form) by functions in $K^p_\th$. This phenomenon becomes 
especially pronounced when passing from $\th$ to 
$$\kbmo:=K^2_\th\cap\bmoa,$$ 
the star-invariant subspace of $\bmoa$, in which case no loss of smoothness 
usually occurs. (Of course, the smoothness property in question should not be too 
strong -- it should not even imply continuity -- if we want a nontrivial inner function 
to have it.) A result to that effect will appear as Corollary \ref{cor:unity} below; we shall 
deduce it from a more general theorem concerning the action of a coanalytic Toeplitz operator $\tg$, 
with $g\in H^1$, on $\kbmo$. However, the very meaning of the expression $\tg f$ (with $f\in\kbmo$) is 
not immediately clear, since the product $f\bar g$ need not be integrable. The following proposition 
will clarify the situation. 

\begin{prop}\label{prop:deftg} Given $f\in\kbmo$ and $g\in H^1$, there 
exists a function $\Phi\in\cap_{0<p<1}H^p$ such that 
$$\|T_{\ov g_n}f-\Phi\|_p\to0$$ 
for every $p\in(0,1)$ and every sequence $\{g_n\}\sb H^2$ with 
$\|g_n-g\|_1\to0$. 
\end{prop}

\par This (obviously unique) function $\Phi$ is then taken to be $\tg f$, 
the image of $f$ under the Toeplitz operator $\tg$. 

\par The proof relies on the following lemma due to Cohn (see Lemma 3.2 
in \cite[p.\,731]{C1}), which in turn results from an application of the 
$(H^1,\bmoa)$ duality. 

\begin{lem}\label{lem:repr} Let $\th$ be inner, and let $f\in\kbmo$. 
Then $f=P_+(\bar z\bar\psi\th)$ for a function $\psi\in H^\infty$. 
Furthermore, $\psi$ may be chosen so that $\|f\|_*=\|\psi\|_\infty$. 
\end{lem}

\par Here and below, $\|\cdot\|_*$ is the dual space norm on $\bmoa$ induced 
by $H^1$. 

\bigskip\noindent{\it Proof of Proposition \ref{prop:deftg}.} Let $f\in\kbmo$, 
$g\in H^1$, and suppose $\{g_n\}$ is a sequence of $H^2$-functions with 
$\|g_n-g\|_1\to0$. We have then 
$$T_{\ov g_n}f=P_+\left(\ov g_nP_+(\bar z\bar\psi\th)\right)
=P_+\left(\bar g_n\bar z\bar\psi\th\right),$$ 
where $\psi$ is related to $f$ as in Lemma \ref{lem:repr}. Now put 
$$\Phi:=P_+\left(\bar g\bar z\bar\psi\th\right).$$ 
This definition makes sense, since $P_+$ is applied to 
an $L^1$-function; besides, it does not depend on the choice of $\psi$. 
(Indeed, if $\psi_1$ and $\psi_2$ are both eligible in the sense of 
Lemma \ref{lem:repr}, then $\psi_1-\psi_2\in\th H^\infty$.) 
And since $P_+$ is a continuous mapping from $L^1$ to every $H^p$ 
with $0<p<1$ (cf. \cite[p.\,128]{G}), we conclude that $\Phi\in H^p$ 
and $\|T_{\ov g_n}f-\Phi\|_p\to0$ for any such $p$. 
\quad\qedsymbol

\medskip Now suppose $X$ is a Banach space of analytic functions on the disk, 
with $X\sb H^1$. We say that $X$ is a {\it $K$-space} if, for each 
$\psi\in H^\infty$, the Toeplitz operator $T_{\bar\psi}$ acts boundedly 
from $X$ to itself, with norm at most $\const\cdot\|\psi\|_\infty$. 
(This is essentially equivalent to saying that $X$ enjoys the so-called 
{\it $K$-property} of Havin. The latter was defined in \cite{H} by the 
formally weaker condition that $T_{\bar\psi}(X)\sb X$, for all 
$\psi\in H^\infty$, but the norm estimate is usually automatic.) 
\par Following \cite{H}, we remark that $X$ will be a $K$-space provided 
it is (isomorphic to) the dual of some Banach space $Y$, consisting of 
analytic functions on $\D$ and satisfying the conditions 

\smallskip (a) $H^\infty\cap Y$ is dense in $Y$, and 

\smallskip (b) for each $\psi\in H^\infty$, the multiplication operator 
$f\mapsto f\psi$ acts boundedly from $Y$ to itself, with norm at most 
$\const\cdot\|\psi\|_\infty$. 

\smallskip\noindent (It is understood that the pairing between $X$ and $Y$ 
is given by $\langle f,g\rangle:=\int_\T f\bar g\,dm$, which is meaningful 
at least for $f\in H^\infty\cap Y$ and $g\in X$.) The Toeplitz operator 
$T_{\bar\psi}:X\to X$ is then the adjoint of the multiplication map in (b), 
which justifies our claim. 

\par As examples of $K$-spaces, we list the following: 
\par$\bullet$\,\,\,$H^p$ with $1<p<\infty$, 
\par$\bullet$\,\,\,the Hardy--Sobolev spaces $H^{p,n}:=\{f\in H^p:
f^{(n)}\in H^p\}$ with $1\le p<\infty$ and $n\ge1$, 
\par$\bullet$\,\,\,$\bmoa$, and more generally, $\bmoa^{(n)}:=\{f\in H^1:
f^{(n)}\in\bmoa\}$ with $n\ge0$, 
\par$\bullet$\,\,\,the Dirichlet-type spaces $\mathcal D_w:=\{f\in H^2:
\sum_{n\ge1}w_n|\widehat f(n)|^2<\infty\}$ associated with nondecreasing 
sequences $w=\{w_n\}$ of positive numbers, 
\par$\bullet$\,\,\,the analytic Besov spaces $B^s_{p,q}$ with $s>0$, 
$p\ge1$, $q\ge1$, and in particular 
\par$\bullet$\,\,\,the classical Lipschitz--Zygmund spaces 
$A^\al:=B^{\al}_{\infty,\infty}$ with $0<\al<\infty$. 
\par\noindent We recall that $B^s_{p,q}$ is defined as the set of those 
analytic $f$ on $\D$ for which the function 
\begin{equation}\label{eqn:bes}
r\mapsto(1-r)^{n-s}\left\|f^{(n)}_r\right\|_p
\end{equation}
is in $L^q$ over the interval $(0,1)$ with respect to the measure 
$dr/(1-r)$; here $n$ is some (any) fixed integer with $n>s$ and 
$f^{(n)}_r(\ze):=f^{(n)}(r\ze)$. 
\par For most of the spaces considered, the $K$-property has been established 
by means of a duality argument, as outlined above. 
We refer to \cite{H}, where this is done for $A^\al$ and some special cases of Hardy--Sobolev 
and Besov spaces; to \cite{Sha1, Sha2} for general $H^{p,n}$ and $B^s_{p,q}$ classes, as well as 
for $\bmoa^{(n)}$; and finally to any of \cite{K, KF, R} in connection with $\mathcal D_w$ spaces. 
\par As further examples of $K$-spaces, we mention $K^p_\th$ ($1<p<\infty$) 
and $\kbmo$. Indeed, for $g\in H^\infty$, one verifies the inclusion 
$\tg(K^p_\th)\sb K^p_\th$ by noting that $K^p_\th$ is the kernel of the 
Toeplitz operator $T_{\bar\th}:H^p\to H^p$, which commutes with $\tg$. 
Then one deduces that $\tg(\kbmo)\sb\kbmo$, recalling that 
$\kbmo=K^2_\th\cap\bmoa$ and $\bmoa$ is a $K$-space. And, of course, the two 
inclusions are accompanied by the natural norm estimates: the norm of $\tg$ 
is in both cases $O(\|g\|_\infty)$, just as it happens for the containing 
spaces $H^p$ ($1<p<\infty$) and $\bmoa$. 
\par The main result of this section is as follows. 

\begin{thm}\label{thm:tb} Let $\th$ be an inner function, 
$g\in H^1$, and let $X$ be a $K$-space. The following 
are equivalent. 

\smallskip{\rm (i)} $\tg$ acts boundedly from $\kbmo$ to $X$. 

\smallskip{\rm (ii)} $\tg$ acts boundedly from $\kinf$ to $X$. 

\smallskip{\rm (iii)} The function 
$$k(z):=\f{\th(z)-\th(0)}z$$ 
satisfies $\tg k\in X$. 

\smallskip Moreover, the operator norms $\|\tg\|_{\kbmo\to X}$ 
and $\|\tg\|_{\kinf\to X}$ are comparable to each other and to 
$\|\tg k\|_X$. 
\end{thm}

\par In most -- perhaps all -- cases of interest, condition 
(iii) above can be further rephrased by saying that $\tg\th\in X$. 
In fact, since $k=T_{\bar z}\th$ and $T_{\bar z}\tg=\tg T_{\bar z}$, 
the implication 
$$\tg\th\in X\implies\tg k\in X$$ 
holds whenever $X$ is a $K$-space. The converse is true provided that 
$1\in X$ and $zX\sb X$; indeed, 
$$\tg\th=\text{\rm const}+z\tg k.$$ 
In particular, we certainly have $\tg k\in X\iff\tg\th\in X$ when $X$ 
is one of our smoothness classes, such as $H^{p,n}$, $B^s_{p,q}$, $A^\al$ 
or $\bmoa^{(n)}$, let alone $H^p$ and $\bmoa$. The theorem then states 
that the inclusion $\tg f\in X$ holds for all $f\in\kbmo$ if and only if 
it holds for $f=\th$. 

\par The next fact is obtained by applying Theorem \ref{thm:tb} with 
$g\equiv1$, in which case $\tg$ reduces to the identity map. 

\begin{cor}\label{cor:unity} Given an inner function $\th$ and a $K$-space 
$X$, one has 
\begin{equation}\label{eqn:chain}
\kbmo\sb X\iff\kinf\sb X\iff k\in X.
\end{equation}
\end{cor}

\par And since the latter condition, $k\in X$, is implied by (and is usually 
equivalent to) saying that $\th\in X$, the nontrivial part of 
\eqref{eqn:chain} amounts to the implication 
\begin{equation}\label{eqn:pup}
\th\in X\implies\kbmo\sb X.
\end{equation}

\bigskip\noindent{\it Proof of Theorem \ref{thm:tb}.} The part 
(i.1)$\implies$(ii.1) is trivially true, as is the inequality 
$$\|\tg\|_{\kinf\to X}\le\|\tg\|_{\kbmo\to X}.$$ 
The part (ii.1)$\implies$(iii.1), along with the estimate 
$$\|\tg\|_{\kinf\to X}\ge\f12\|\tg k\|_X,$$ 
is also obvious, since $k\in\kinf$ and $\|k\|_\infty\le2$. 
\par What remains to be proved is the implication (iii.1)$\implies$(i.1) 
and its quantitative version 
\begin{equation}\label{eqn:normest}
\|\tg\|_{\kbmo\to X}\le\text{\rm const}\cdot\|\tg k\|_X.
\end{equation}
To this end, we fix $f\in\kbmo$ and then invoke Lemma \ref{lem:repr} 
to find a function $\psi\in H^\infty$ such that 
$$f=T_{\bar z\bar\psi}\th,\qquad\|f\|_*=\|\psi\|_\infty.$$ 
Using the fact that coanalytic Toeplitz operators commute (and moreover, 
$T_{\bar a}T_{\bar b}=T_{\bar a\bar b}$ whenever $a$, $b$ and $ab$ are 
$H^1$-functions such that the operators involved are all well-defined), 
we obtain 
\begin{equation}\label{eqn:slon}
\tg f=\tg T_{\bar z\bar\psi}\th=T_{\bar\psi}\tg T_{\bar z}\th
=T_{\bar\psi}\tg k.
\end{equation}
Finally, we recall that $X$ is a $K$-space to get 
\begin{equation*}
\begin{aligned}
\|\tg f\|_X&\le\|T_{\bar\psi}\|_{X\to X}\|\tg k\|_X\\
&\le\text{\rm const}\cdot\|\psi\|_\infty\|\tg k\|_X\\
&=\text{\rm const}\cdot\|f\|_*\|\tg k\|_X,
\end{aligned}
\end{equation*} 
which readily implies \eqref{eqn:normest}.\quad\qedsymbol

\medskip Finally, we supplement Theorem \ref{thm:tb} with the following result. 

\begin{prop}\label{prop:lpbound} Let $\th$, $g$ and $k$ be 
as above. The operator $\tg$ acts boundedly from $\kbmo$ to itself 
if and only if $\tg k\in\bmoa$. In this case we also have 
$$\|\tg f\|_p\le C_p\|\tg k\|_*\|f\|_p,\qquad 1<p<\infty,$$ 
for all $f\in\kinf$, so that $\tg$ extends to a bounded operator 
on $K^p_\th$. 
\end{prop}

\par This might be compared to the \lq\lq $T(1)$-" 
and/or \lq\lq $T(b)$-theorem" of David, Journ\'e and Semmes 
(cf. \cite[Chapter 5]{DynEnc} or \cite[Chapter VII]{S}), results 
that provide boundedness criteria for certain singular integral operators 
on $L^p$. Just as in those theorems, we only have to test the operator 
on a single function. We also remark that the assumption $\tg k\in\bmoa$ 
can be rewritten as $\tg\th\in\bmoa$, and a sufficient condition for this 
to happen is that 
$$\sup\{|g(z)|:\,z\in\Om(\th,\eps)\}<\infty$$
for some $\eps\in(0,1)$, where $\Om(\th,\eps)$ is the sublevel set defined 
by \eqref{eqn:defsublevel}. A proof of this last assertion can be found 
in \cite{DAJM}. 

\medskip\noindent{\it Proof of Proposition \ref{prop:lpbound}.} The first 
statement, concerning the action of $\tg$ on $\kbmo$, is obtained by 
applying Theorem \ref{thm:tb} with $X=\bmoa$ (or $X=\kbmo$). 
\par Now suppose $\tg k\in\bmoa$, and let $1<p<\infty$. Given a function 
$f\in\kinf$, put $\psi:=\bar z\bar f\th(=\widetilde f)$ and note that 
$\psi\in H^\infty$. We have then 
$$f=\bar z\bar\psi\th=P_+\left(\bar z\bar\psi\th\right)
=T_{\bar z\bar\psi}\th,$$ 
and so \eqref{eqn:slon} remains in force. Setting $h:=\tg k$ and making 
use of the elementary identity 
%\begin{equation}\label{eqn:elemid}
$$\ov{P_+F}=zP_-(\bar z\bar F),\qquad F\in L^1,$$
%\end{equation} 
we can rewrite the resulting equality from \eqref{eqn:slon} as 
$$\tg f=T_{\bar\psi}h=\bar z\ov{H_{\bar z\bar h}\psi}.$$ 
In view of Nehari's theorem (see, e.\,g., \cite[Part B, Chapter 1]{N}), 
the assumption that $h$, and hence $zh$, is in $\bmoa$ implies 
that the Hankel operator $H_{\bar z\bar h}$ acts boundedly 
from $H^p$ to $\ov{H^p_0}$, with norm not exceeding $C_p\|h\|_*$. 
Consequently, 
$$\|\tg f\|_p=\left\|H_{\bar z\bar h}\psi\right\|_p\le C_p\|h\|_*\|\psi\|_p
=C_p\|h\|_*\|f\|_p,\qquad f\in\kinf.$$ 
Finally, since $\kinf$ is dense in $K^p_\th$ (indeed, $\kinf$ contains 
the family of reproducing kernels for $K^2_\th$), we conclude that $\tg$ 
extends to a bounded operator on $K^p_\th$, with the same norm. 
The proof is complete.\quad\qedsymbol

\end{document}